\documentclass[11pt]{article}
\usepackage{amsmath,amsthm,amsfonts,amssymb,amscd, amsxtra, mathrsfs}
\usepackage{url}
\usepackage[margin=2.7 cm,nohead]{geometry}
\usepackage[ruled]{algorithm2e} 
\usepackage{algorithmic}
\usepackage{color}
\usepackage[active]{srcltx}
\usepackage{verbatim}
\usepackage{url}
\usepackage{cite}
\usepackage{enumerate}
\newtheorem{theorem}{Theorem}
\newtheorem{lemma}[theorem]{Lemma}

\newtheorem{proposition}[theorem]{Proposition}

\newtheorem{remark}{Remark}

\DeclareMathOperator{\argmin}{argmin}

\newcommand{\R}{\mathbb R}

\newcommand{\bs}{\left(\begin{smallmatrix}}
\newcommand{\es}{\end{smallmatrix}\right)}
\begin{document}
\title{Reducing the projection onto the monotone extended second-order cone to the pool-adjacent-violators algorithm of isotonic regression}
\author{O. P. Ferreira \thanks{IME, Universidade Federal de Goias, Avenida Esperança, s/n, Campus Samambaia,  Goi\^ania, GO, 74690-900, Brazil {\tt orizon@ufg.br}.  The authors was supported in part by  CNPq grants 305158/2014-7 and 302473/2017-3.}
\and
Y. Gao  \thanks{School of Mathematics, University of Birmingham, Watson Building, Edgbaston, Birmingham B15 2TT, United Kingdom,  {\tt YXG713@bham.ac.uk}, {\tt s.nemeth@bham.ac.uk}.} 
\and 
S. Z. N\'emeth   \footnotemark[2]
 }
 
\maketitle

\begin{abstract}
This paper introduces the monotone extended second-order cone (MESOC), which is related to 
the monotone cone and the second-order cone. Some properties of the MESOC are presented and
its dual cone is computed. Projecting onto the MESOC is reduced to the pool-adjacent-violators
algorithm (PAVA) of isotonic regression. An application of MESOC to portfolio optimisation is
provided. Some broad descriptions of possible MESOC-regression models are also outlined. \\

\end{abstract}

\section{Introduction} \
The purpose of this paper is to introduce a new second-order cone, which we call the  monotone extended second-order cone. Some properties
of the MESOC  are studied and formulas for projecting onto  it are presented. We will follow the ideas used in \cite{FerreiraNemeth2018} for
projecting onto a non-monotone extension of the second-order cone. It is worth to note that the projection in this paper is considerably more
difficult to find, because it is partly based on projecting onto the monotone nonnegative cone, which is a nontrivial problem compared to the
projection onto the nonnegative orthant, see \cite{mair2009isotone2009,nemeth2012projec}. The definition of the MESOC relates two well-known
cones, namely, the monotone cone and a second-order cone, known  as  Lorentz  cone.    The monotone cone  has connections  with  the isotonic regression problem, in fact it is the constraint  set of this problem,  see for example \cite{BestChakravarti1990}. This cone arises in statistics and has also connections with finance \cite{le2016application}.  In  \cite{2Nemeth2016}  some   properties of the weighted version of the monotone cone have been also considered. The Lorentz cone is an important object in theoretical physics, and it is commonly used in optimization, a good survey paper with a wide range of applications of second-order cone programming is \cite{MR1655138}. Various  connections of second-order cone programming and second-order cone complementarity problem with physics, mechanics, economics, game theory, robotics, optimization and neural networks have been considered in
\cite{MR3158056,MR2377196,MR2116450,MR3010551,MR2925039,MR2568432,MR2522815,KCY2011,MR2179239}. 

The structure of the paper is as follows: In Section~\ref{sec:preliminaries} we fix the notations and the terminology used throughout the paper.
In Section~\ref{sec:defdual}  we introduce  the MESOC and  compute its  dual cone, and in  Section~\ref{sec:compl}  we find   the complementarity
set of the MESOC.  The  formulas for projecting onto the pair of mutually dual monotone extended second-order cones are derived in Section~\ref{sec:proj}. 
In Section~\ref{eq:apo} we have presented an application of the MESOC to portfolio 
optimisation via a conic optimization problem related to the mean-absolute deviation model 
\cite{konno1991mean}.     
Finally, we make some remarks in the last section, including some broad descriptions about 
how could the projection onto the MESOC occur directly in modelling some practical 
problems.   
\section{Preliminaries} \label{sec:preliminaries}
Here, we recall some notations, definitions, and basic properties of convex cones and projections onto it.  Let $\ell,m,p,q$ be positive integers such that $m=p+q$. We identify the  vectors of ${\mathbb R}^\ell$ with $\ell\times 1$ matrices with real
entries. The scalar product in ${\mathbb R}^\ell$ and the  corresponding norm  are  defined, respectively, by
${\mathbb R}^\ell\times{\mathbb R}^\ell\ni (x,y)\mapsto\langle x,y \rangle:=x^\top y\in{\mathbb R}$ and  $ {\mathbb R}^\ell\ni x\mapsto\|x\|:=\sqrt{\langle x,x\rangle}\in{\mathbb R}$. 
The equality $\langle x,y\rangle=0$ is denoted by $x\perp y$.
We identify the elements of ${\mathbb R}^p\times{\mathbb R}^q$ with the elements of ${\mathbb R}^m$ through the correspondence ${\mathbb R}^p\times{\mathbb R}^q\ni(x,y)\mapsto (x^\top,y^\top)^\top$. Through this
identification the scalar product in ${\mathbb R}^p\times{\mathbb R}^q$ is defined by $\langle (x,y),(u,v)\rangle:=\langle x,u\rangle+\langle
y,v\rangle$.  A closed set ${\cal K}\subseteq{\mathbb R}^\ell$ with nonempty interior is
called a \emph{proper cone} if ${\cal K}+{\cal K}\subseteq {\cal K}$, ${\cal K}\cap(-{\cal K})=\{0\}$ and $\lambda {\cal K}\subseteq {\cal K}$, for any $\lambda$ positive real number. The \emph{dual
cone} of a proper cone $K\subseteq{\mathbb R}^\ell$ is a proper cone defined by ${\cal K}^*:=\{x\in{\mathbb R}^\ell~:~\langle x,y\rangle\ge0,\mbox{ }\forall y\in {\cal K}\}$.  For a proper cone ${\cal K}\in{\mathbb R}^\ell$,  the \emph{complementarity set} of ${\cal K}$ is defined  by
$C({\cal K}):=\left\{(x,y)\in {\cal K}\times {\cal K}^*:~x\perp y\right\}$. 
 Let $C\in{\mathbb R}^\ell$ be a closed convex set. The projection mapping $P_C\colon{\mathbb R}^\ell\to{\mathbb R}^\ell$ onto $C$ is defined by $P_C(x):=\argmin\{\|x-y\|:y\in C\}$, which is piecewise linear whenever $C$ is a polyhedral cone; see \cite[Definition 4.1.3 and Proposition~4.1.4]{FacchineiPang2003-I}.  We recall here Moreau's decomposition theorem  \cite{MR0139919} (stated here for proper cones only):
\begin{theorem} \label{th:mt}
	Let ${\cal K}\subseteq{\mathbb R}^\ell$ be a proper cone, ${\cal K}^*$ its dual cone and $z\in{\mathbb R}^\ell$. Then, the following two statements are equivalent:
	\begin{enumerate}
		\item[(i)] $z=x-y$ and $(x,y)\in C({\cal K})$,
		\item[$(ii)$] $x=P_{\cal K}(z)$ and $y=P_{{\cal K}^*}(-z)$.
	\end{enumerate}
\end{theorem}
In particular,  Theorem~\ref{th:mt} implies  that
$$
 P_{\cal K}(z)\perp P_{{\cal K}^*}(-z), \qquad  z=P_{\cal K}(z)-P_{{\cal K}^*}(-z).
$$
For  $z\in{\mathbb R}^\ell$ we denote  $z=(z_1,\dots,z_\ell)^\top$.  Denote by  ${\mathbb R}^\ell_+=\{x\in{\mathbb R}^\ell~:~ x\ge0\}$ the nonnegative orthant. The proper cone ${\mathbb R}^\ell_+$ is self-dual, i.e., ${\mathbb R}^\ell_+=({\mathbb R}^\ell_+)^*$. For a real number $\alpha\in{\mathbb R}$ denote $\alpha^+:=\max(\alpha,0)$ and $\alpha^-:=\max(-\alpha,0)$. 
For a vector $z\in{\mathbb R}^\ell$ denote $z^+:=(z_1^+,\dots,z_\ell^+)^\top$,
$z^-:=(z_1^-,\dots,z_\ell^-)^\top$ and  $|z|:=(|z_1|,\dots,|z_\ell|)^\top$.
Therefore, $z^+=P_{{\mathbb R}^\ell_+}(z)$,  $z^-=P_{{\mathbb R}^\ell_+}(-z)$,
$z=z^+-z^-$ and $|z|=z^++z^-$. In particular, we denote $P_{\cal K}(z)^+=x^+$ and
$P_{\cal K}(z)^-=x^-$, where ${\cal K}\subseteq\mathbb R^{\ell}$ is a proper cone
and $x=P_{\cal K}(z)$. Thus, $P_{\cal K}(z)=P_{\cal K}(z)^+-P_{\cal K}(z)^-$. Without 
leading to any confusion,
depending on the context, we will denote by $0$ the vector in ${\mathbb R}^\ell$ or a scalar zero and by  $e^i\in \mathbb{R}^p$  the $i$-th
canonical unit vector, i.e., the vector with all coordinates $0$ except the $i$-th coordinate which is $1$. The  {\it monotone  cone} $ \R^p_{\geq}$ is defined as follows:
\begin{equation} \label{eq:defmca}
 \R^p_{\geq}:=\left\{x\in\mathbb{R}^p:~x_1\geq x_2\geq\cdots\geq x_p\right\}.
\end{equation}
Let $j\in\{1,...,p-1\}$. To simplify the notations we define
$$
e^{1:j}:=e^1+\cdots+e^j=(\underbrace{1, \ldots, 1}_{j\rm\ times},\underbrace{0, \ldots, 0}_{p-j\rm\ times})\in \R^p, \qquad  e:=e^1+\cdots+e^p=(\underbrace{1,\ldots,1}_{p\rm\ times}) \in \R^p.
$$
The {\it dual} of the cone $ \R^p_{\geq}$  is  given by 
\begin{equation} \label{eq:defdmca}
(\R^p_{\geq})^*:=\left\{y\in\mathbb{R}^p:~ \left\langle y, e^{1:j}\right\rangle \geq 0, ~ j=1,\ldots,p-1,~ \left\langle y, e\right\rangle=0\right\}.
\end{equation}
The {\it monotone nonnegative cone}, is defined by
\begin{equation} \label{eq:defmnc}
 \R^p_{\geq+}:=\left\{x\in\mathbb{R}^p:~x_1\geq x_2\geq\cdots\geq x_p\geq 0\right\}.
\end{equation}
The {\it dual} of  the 
cone  $\R^p_{\geq+}$ is  given by 
\begin{equation} \label{eq:defdmnc}
(\R^p_{\geq+})^*:=\left\{y\in\mathbb{R}^p:~ \left\langle y, e^{1:j}\right\rangle \geq 0, ~ j=1,\ldots,p-1,~ \left\langle y, e\right\rangle\geq 0\right\}.
\end{equation}
\section{The monotone  extended second-order cone} \label{sec:defdual}
In this section we introduce  the  monotone  extended second-order cone, which 
generalize the well known  Lorentz cone.  We also compute the dual cone of the monotone  extended second-order cone.  The {\it monotone extended
second-order cone} $ {\mathcal L}_{p,q}\subseteq \mathbb{R}^m:=\mathbb{R}^{p+q}$  is  defined as follows:
\begin{equation} \label{defmesoc}
{\mathcal L}_{p,q}:=\big\{(x,u)\in\mathbb{R}^p\times\mathbb{R}^q:~x_1\geq x_2\geq\cdots\geq x_p\geq\|u\|\big\}.
\end{equation}
\begin{remark} \label{re:Lorentz}
If $p,q\ge 1$, then the cone ${\mathcal L}_{p,q}$ is a proper cone. Letting  $p=1$ in \eqref{defmesoc}, the cone ${\mathcal L}_{p,q}$  becomes 
${\mathcal L}_{1,p}=\{(t,u)\in\mathbb R\times\mathbb R^q:t\ge\|u\|\}$, which  is the  second-order cone in $\mathbb R^{1+q}\equiv\mathbb
R\times\mathbb R^q$  known as Lorentz cone.  The cone ${\mathcal L}_{p,q}$ is polyhedral, if and only if $q=0$ or $q=1$. If $q=0$, then the cone
${\mathcal L}_{p,q}$ becomes the monotone nonnegative cone $\mathbb R^p_{\ge+}$ defined in 
\eqref{eq:defmnc}.
\end{remark}
Before proceeding with our presentation, let us state {\it Abel's partial summation formula} 
that  will be useful to study the properties of the MESOC:
\begin{equation} \label{eq:fe}
  \langle x,y\rangle= \sum_{i=1}^{p-1}(x_i-x_{i+1})\sum_{j=1}^{i}y_j+x_p\sum_{i=1}^{p}y_i, \qquad \forall x, y\in \mathbb{R}^p.
\end{equation}
Interesting  applications of this formula  can be found in \cite{Niculescu2017, Niculescu2018}.  Next  we present the  dual cone of the MESOC.  
\begin{proposition}
The dual cone ${\mathcal L}_{p,q}^*$  of the monotone extended second-order cone ${\mathcal L}_{p,q}$ is
\begin{equation} \label{dualmesoc}
{\mathcal L}_{p,q}^*:=\left\{(y,v)\in\mathbb{R}^p\times\mathbb{R}^q:~ \left\langle y, e^{1:j}\right\rangle \geq 0, ~ j=1,\ldots,p-1,~ \left\langle y, e\right\rangle\geq\|v\|\right\}.
\end{equation}
\end{proposition}
\begin{proof}
To simplify the notations, denote by $M$ the right hand side of \eqref{dualmesoc}. Our task is to prove that $M={\mathcal L}_{p,q}^*$, this will be done by proving that $M\subseteq {\mathcal L}_{p,q}^*$ and ${\mathcal L}_{p,q}^*\subseteq M$. We proceed to prove the first inclusion, for that   take $(y,v)\in M$.  The  definition of $M$ implies
\begin{equation} \label{eq:fsctop}
\left\langle y, e^{1:i}\right\rangle=\sum_{j=1}^{i}y_j\geq 0, \quad i=1, \ldots, p-1,  \quad   \qquad \left\langle y, e\right\rangle=\sum_{i=1}^{p}y_i\geq\|v\|.
\end{equation}
Let  $(x,u)\in {\mathcal L}_{p,q}$ be arbitrary.  The definition of ${\mathcal L}_{p,q}$ implies   $x_1-x_{2}\geq 0, \ldots, x_{p-1}-x_{p}\geq 0$,   and $x_p\geq \|u\|$, which together  with \eqref{eq:fe} and \eqref{eq:fsctop} yield
$$
  \langle x,y\rangle=\sum_{i=1}^{p-1}(x_i-x_{i+1})\sum_{j=1}^{i}y_j+x_p\sum_{i=1}^{p}y_i\geq  \|u\|\|v\|.
$$
Therefore,  the last inequality  and   Cauchy's inequality  imply
$$
\langle (x,u),(y,v)\rangle=\langle x,y\rangle+ \langle u,v\rangle \geq\|u\|\|v\|+ \langle u,v\rangle\geq 0, 
$$
which proves  the inclusion  $M\subseteq {\mathcal L}_{p,q}^*$.  To prove the second inclusion,  take  $(y,v)\in {\mathcal L}_{p,q}^*$.  First
note  that  $\left(e^{1:j}, 0\right)\in {\mathcal L}_{p,q}$. Thus, since $(y,v)\in {\mathcal L}_{p,q}^*$,  we have  $\left\langle\left (e^{1:j}, 0\right), (y,v)\right\rangle \geq 0$, for all $j=1,2,\ldots,p-1$,  which implies
\begin{equation} \label{eq:fe0t1}
\left\langle y, e^{1:j}\right\rangle\geq 0, \qquad  \forall ~j=1,2,\ldots,p-1.
\end{equation}
To proceed, first assume  $v=0$. Since $(e,0)\in {\mathcal L}_{p,q}$ and $ (y,0)\in {\mathcal L}_{p,q}^*$,  we have
\begin{equation} \label{eq:fe0}
\left\langle y, e\right\rangle \geq 0=\|v\|.
\end{equation}
Now,   assume $v\neq 0$. Since  $(\|v\|e,-v) \in {\mathcal L}_{p,q}$ and $(y,v)\in {\mathcal L}_{p,q}^*$, we obtain that  $\langle (\|v\|e,-v),
(y,v)\rangle \geq 0$, which implies    $\|v\|\left\langle y, e\right\rangle-\|v\|^2 \geq 0$. Thus, due  to $v\neq 0$,  we have $\left\langle y,
e\right\rangle-\|v\|\geq 0$. Therefore, the last inequality together with \eqref{eq:fe0} 
imply that
\begin{equation} \label{eq:fe0s1}
\left\langle y, e\right\rangle\geq \|v\|,
\end{equation}
for all $(y,v)\in {\mathcal L}_{p,q}^*$.   Hence,  it follows from   \eqref{eq:fe0t1} and  \eqref{eq:fe0s1}  that $(y,v)\in M$. Therefore, we conclude that  ${\mathcal L}_{p,q}^*\subseteq M$.
Since $M\subseteq {\mathcal L}_{p,q}^*$ and ${\mathcal L}_{p,q}^*\subseteq M$, we have ${\mathcal L}_{p,q}^* = M$.
\end{proof}
\begin{remark}
Letting  $p=1$ in \eqref{dualmesoc}, there are no  inequalities,  for $j=1,\ldots ,p-1$, because  $p-1=0$. Thus, the  cone ${\mathcal L}_{p,q}^*$
becomes the  Lorentz cone ${\mathcal L}_{1,p}$ (see also Remark~\ref{re:Lorentz}).
\end{remark}

\section{The complementarity set} \label{sec:compl}
After finding the dual of the monotone extended second-order cone, we want to find the complementarity set of this cone. In order to find the complementarity set, we need  two    inequalities    introduced in the next lemma.
\begin{lemma}\label{Le:p3}
Let  $(x,u) \in {\mathcal L}_{p,q}$ and $(y,v) \in {\mathcal L}_{p,q}^*$. Then,
\begin{equation} \label{eq:cier}
\langle x,y \rangle\geq\|u\|\left\langle y, e\right\rangle\geq\|u\|\|v\|.
\end{equation}
\end{lemma}
\begin{proof}
Since $(x,u) \in {\mathcal L}_{p,q}$,  we have $x_1\geq x_2\geq\cdots\geq x_p\geq\|u\|$. Thus,  letting  $ 0\in \mathbb{R}^q$, we have $ (x-\|u\|e, 0)\in {\mathcal L}_{p,q}$. Considering that $(y,v) \in {\mathcal L}_{p,q}^*$, the definition of ${\mathcal L}_{p,q}^*$ yields
$$
0\leq \left\langle (x-\|u\|e, 0),(y,v)\right\rangle=\langle x, y\rangle-\|u\|\langle y, e\rangle.
$$
which implies the first inequality in \eqref{eq:cier}.   Since  $(y,v) \in {\mathcal L}_{p,q}^*$,  we have $\left\langle y, e\right\rangle\geq\|v\|$, from where the second  inequality  in \eqref{eq:cier}  follows.
\end{proof}
In the next proposition we presents  some  relationships  of the monotone extended second-order cone with the monotone nonnegative cone. Since its proof is an immediate consequence  of \eqref{defmesoc}, \eqref{dualmesoc}, \eqref{eq:defmnc} and \eqref{eq:defdmnc},  it will be omitted.
\begin{proposition}\label{prop:cseq}
Let $(x,u), (y,v)\in\R^p\times\R^q$. Then, there hold:
\begin{itemize}
  \item [\textrm{$(i)$}]   $(x,u)\in {\mathcal L}_{p,q}$  if and only if $x-\|u\|e\in \R_{\geq + }^p$.
  \item [\textrm{$(ii)$}]   $(y,v)\in {\mathcal L}_{p,q}^*$ if and only if  $y-\|v\|e^p\in (\R_{\geq + }^p)^*$.
\end{itemize}
\end{proposition}
By using  Lemma~\ref{Le:p3} and  Proposition~\ref{prop:cseq}, next we determine   the  complementarity set of ${\mathcal L}_{p,q}$.
\begin{proposition}\label{prop:csa}
Let $x,y\in{\mathbb R}^p$ and $u,v\in{\mathbb R}^q\setminus\{0\}$.Then  $(x,u,y,v):=((x,u),(y,v))\in C({\mathcal L}_{p,q})$ if and only if   $x_p = \|u\|$, $\langle y, e\rangle=\|v\|$,  $\langle u,v\rangle=-\|u\|\|v\|$, and  $\left(x-\|u\|e,y-\|v\|e^p\right)\in C(\R_{\geq + }^p)$.
\end{proposition}
\begin{proof}
Take $(x,u,y, v)\in C({\mathcal L}_{p,q})$. The definition of $C({\mathcal L}_{p,q})$ implies  $(x, u)\in {\mathcal L}_{p,q}$, $(y, v)\in
{\mathcal L}_{p,q}^*$ and $\langle (x,u),(y,v)\rangle=0$.   Since $(x, u)\in {\mathcal L}_{p,q}$ and $(y, v)\in {\mathcal L}_{p,q}^*$,
Proposition~\ref{prop:cseq} implies that      $x-\|u\|e\in \R_{\geq + }^p$ and  $y-\|v\|e^p\in (\R_{\geq + }^p)^*$.  Furthermore, the condition
$\langle (x,u),(y,v)\rangle=0$, Lemma~\ref{Le:p3} and the Cauchy inequality imply that
$$
0=\langle x, y\rangle+ \langle u, v\rangle\geq \|u\|\langle y, e\rangle+  \langle u, v\rangle\geq \|u\|\|v\|+ \langle u, v\rangle\geq 0.
$$
Thus,  $\langle x,y\rangle=\|u\|\langle y, e\rangle$, $ \|u\|\langle y, e\rangle=\|u\|\|v\|$ and  $\langle u,v\rangle=-\|u\|\|v\|$. Moreover, taking into account that  $u\neq 0$,  we also have $\langle y, e\rangle=\|v\|$. Hence, using \eqref{eq:fe}, we conclude that
\begin{align*}
\left(\|u\|-x_p\right)\|v\|=\left(\|u\|-x_p\right)\langle y, e\rangle =\sum_{i=1}^{p-1}(x_i-x_{i+1})\sum_{j=1}^{i}y_j.
\end{align*}
Since $(x,u)\in {\mathcal L}_{p,q}$ and $(y,v)\in {\mathcal L}_{p,q}^*$, the left hand side and the right hand side of the last equality have
opposite signs. Hence, they must be $0$. In particular $\left(\|u\|-x_p\right)\|v\|=0$. Thus,  due to  $v\neq 0$,  we conclude that $x_p=\|u\|$.  On the other hand,
$$
\langle x-\|u\|e,y-\|v\|e^p\rangle= \langle x, y\rangle- \|u\|\langle y, e\rangle -  x_p\|v\|+\|u\|\|v\|,
$$
which taking into account that $\langle x,y\rangle=\|u\|\langle y, e\rangle$ and $x_p=\|u\|$, yields $\langle x-\|u\|e,y-\|v\|e^p\rangle=0$. Hence,
$(x-\|u\|e,y-\|v\|e^p)\in C(\R_{\geq + }^p)$, which  concludes  the proof of necessity.

 Reciprocally, assume that   $x_p = \|u\|$,  $\langle y, e\rangle=\|v\|$, $\langle u,v\rangle=-\|u\|\|v\|$ and  $(x-\|u\|e,y-\|v\|e^ p)\in C(\R_{\geq + }^p)$. First note that $x-\|u\|e\in \R_{\geq + }^p$, $y-\|v\|e^p\in (\R_{\geq + }^p)^*$ and $\langle x-\|u\|e,y-\|v\|e^p\rangle=0$.  Since $x-\|u\|e\in \R_{\geq + }^p$ and  $y-\|v\|e^p\in (\R_{\geq + }^p)^*$, Proposition~\ref{prop:cseq} implies$(x, u)\in {\mathcal L}_{p,q}$ and $(y, v)\in {\mathcal L}_{p,q}^*$. On the other hand, the equality   $\langle x-\|u\|e,y-\|v\|e^ p\rangle=0$ implies that
$$
 \langle x, y\rangle- \|u\|\langle y, e\rangle -  x_p\|v\|+\|u\|\|v\|=0.
$$
 Thus, due to $x_p = \|u\|$, we conclude that  $ \langle x, y\rangle= \|u\|\langle y, e\rangle$.    Hence, also using $\langle u,v\rangle=-\|u\|\|v\|$ and $\langle y, e\rangle=\|v\|$, we obtain
$$
 \langle (x,u),(y,v)\rangle=\langle x,y\rangle+ \langle u,v\rangle=  \|u\|\langle y, e\rangle-\|u\|\|v\|=\|u\|\left(\langle y, e\rangle-\|v\|\right)=0.
$$
Therefore, $(x,u,y, v)\in C({\mathcal L}_{p,q})$.
\end{proof}
\section{Projection onto monotone  extended second-order cone} \label{sec:proj}
The aim of this section is to present the formulas for projecting onto the pair of mutually dual monotone extended second-order cone. For that we need a preliminary result. 
\begin{lemma} \label{le:u0v0}
Let $(z,w)\in\R^p\times\R^q$. If $P_{{\mathcal L}_{p,q}}(z,w)=(x,u)$ and $P_{{\mathcal L}_{p,q}^*}(-z,-w)=(y,v)$,   then  the following statements hold:
\begin{itemize}
  \item [\textrm{$(i)$}]   $\langle P_{(\R_{\geq+}^p)^*}(-z), e\rangle \geq \|w\|$  if and only if $u=0$;
  \item [\textrm{$(ii)$}]  $P_{\R_{\geq+}^p}(z)_p \geq \|w\|$ if and only if $v=0$.
   \item [\textrm{$(iii)$}] $\langle P_{(\R_{\geq+}^p)^*}(-z), e\rangle < \|w\|$ and $P_{\R_{\geq+}^p}(z)_p < \|w\|$  if and only if $u\neq 0$ and $v\neq 0$.
\end{itemize}
\end{lemma}
\begin{proof}
To prove  
item (i), we first assume that  $u=0$. Considering that  $P_{{\mathcal L}_{p,q}}(z,w)=(x,0)$ and  $P_{{\mathcal L}_{p,q}^*}(-z,-w)=(y,v)$,
Theorem~\ref{th:mt} for ${\mathcal L}_{p,q}$ implies that $(x,0)\in {\mathcal L}_{p,q}$, $(y,v)\in {\mathcal L}_{p,q}^*$,  $\langle (x,0),
(y,v)\rangle=0$ and  $(z, w)=(x,0)-(y,v)$. Hence, we have  $x\in \R_{\geq+}^p$, $y\in (\R_{\geq+}^p)^*$,  $\langle y,e\rangle \geq \|v\|$,
$\langle x,y\rangle =0$,  $z=x-y$ and  $w=-v$.  Hence, by applying Theorem~\ref{th:mt} for
$\R_{\geq+}^p$, we obtain that   $x=P_{\R_{\geq+}^p}(z)$ and $y=P_{(\R_{\geq+}^p)^*}(-z)$. Since $w=-v$ and $\langle y,e\rangle \geq \|v\|$, we
have that $\langle P_{(\R_{\geq+}^p)^*}(-z), e\rangle \geq \|w\|$.  Conversely, suppose that $\langle  P_{(\R_{\geq+}^p)^*}(-z), e\rangle \geq
\|w\|$.  First note that $(P_{\R_{\geq+}^p}(z),0)\in {\mathcal L}_{p,q}$ and, using $\langle  P_{(\R_{\geq+}^p)^*}(-z), e\rangle \geq \|w\|$, we
have  $(P_{(\R_{\geq+}^p)^*}(-z),-w)\in {\mathcal L}_{p,q}^*$.  Moreover,  we conclude that
$(P_{\R_{\geq+}^p}(z),0,P_{(\R_{\geq+}^p)^*}(-z),-w)\in C({\mathcal L}_{p,q})$ and  $(z, w)=(P_{\R_{\geq+}^p}(z), 0)-(P_{(\R_{\geq+}^p)^*}(-z),
-w)$.  Hence, by applying  Theorem~\ref{th:mt} for ${\mathcal L}_{p,q}$, we have $P_{{\mathcal L}_{p,q}}(z,w)= (P_{\R_{\geq+}^p}(z),0)$ and $P_{{\mathcal L}_{p,q}^*}(-z,-w)=(P_{(\R_{\geq+}^p)^*}(-z),-w)$. Therefore,  $u=0$.

To prove item (ii), we first assume that  $v=0$.  Considering that  $P_{{\mathcal L}_{p,q}}(z,w)=(x,u)$ and  $P_{{\mathcal
L}_{p,q}^*}(-z,-w)=(y,0)$, Theorem~\ref{th:mt} for ${\mathcal L}_{p,q}$ implies that $(x,u)\in {\mathcal L}_{p,q}$, $(y,0)\in {\mathcal
L}_{p,q}^*$,  $\langle (x,u), (y,0)\rangle=0$ and  $(z, w)=(x,u)-(y,0)$. Hence, we have  $x\in \R_{\geq+}^p$, $y\in (\R_{\geq+}^p)^*$,  $x_p \geq
\|u\|$,  $\langle x,y\rangle =0$,  $z=x-y$ and  $w=u$. Thus, by applying Theorem~\ref{th:mt} for $\R_{\geq+}^p$, we obtain that
$x=P_{\R_{\geq+}^p}(z)$ and $y=P_{(\R_{\geq+}^p)^*}(-z)$. Since $w=u$ and $x_p \geq \|u\|$, we have that $P_{\R_{\geq+}^p}(z)_p \geq \|w\|$.
Conversely, assume that $P_{\R_{\geq+}^p}(z)_p \geq \|w\|$.  Note that $(P_{(\R_{\geq+}^p)^*}(-z),0)\in {\mathcal L}_{p,q}^*$ and, using
$P_{\R_{\geq+}^p}(z)_p \geq \|w\|$, we have  $(P_{\R_{\geq+}^p}(z),w)\in {\mathcal L}_{p,q}$.  Moreover,
$(P_{\R_{\geq+}^p}(z),w,P_{(\R_{\geq+}^p)^*}(-z),0)\in C({\mathcal L}_{p,q})$ and  $(z,w)=(P_{\R_{\geq+}^p}(z), w)-(P_{(\R_{\geq+}^p)^*}(-z), 0)$.
Hence, by applying  Theorem~\ref{th:mt} for the cone  ${\mathcal L}_{p,q}$, we have $P_{{\mathcal L}_{p,q}}(z,w)= (P_{\R_{\geq+}^p}(z),w)$ and $P_{{\mathcal L}_{p,q}^*}(-z,-w)=(P_{(\R_{\geq+}^p)^*}(-z),0)$. Therefore,  $v=0$.

Item (iii) is an immediate consequence of items $(i)$ and $(ii)$.
\end{proof}
The next lemma  is essential for reducing the projection onto the MESOC to isotonic 
regression. 
\begin{lemma}\label{le:mainlema}
Let $(z,w)\in\R^p\times\R^q$ such that $w\neq 0$.  Assume that  $ P_{{\mathcal L}_{p,q}}(z,w)=(x, \beta w)$ for some $x\in \R^p$ and   $\beta >0$. Then, 
\begin{equation} \label{eq:mle}
 P_{\R_{\geq+}^{p+1}}(z, \|w\|)=(x, \beta \|w\|).
\end{equation}
\end{lemma}
\begin{proof}
Suppose by contradiction that \eqref{eq:mle} does not hold.   Hence,  $P_{\R_{\geq+}^{p+1}}(z,\|w\|)=(y,\mu\|w\|)$ for some $y\in \R^p$ and
$\mu\geq 0$ with  $(y, \mu\|w\|) \ne (x,\beta\|w\|)$. Let $u:=\beta w$ and   $v:=\mu w$. Then, we have $P_{\R_{\geq+}^{p+1}}(z,\|w\|)=(y,\|v\|)$
and consequently $(y,v)\in {\mathcal L}_{p,q}$. Hence,  due to  
\[\R_{\geq+}^{p+1}\ni (x,\|u\|)\ne (y,\|v\|)=P_{\R_{\geq+}^{p+1}}(z,\|w\|),\] we 
obtain that 
$$
\|z-x\|^2+(\|w\|-\|u\|)^2>\|z-y\|^2+(\|w\|-\|v\|)^2.
$$
 Because $w$, $u$ and $u$  are collinear vectors with the same orientation, the last inequality implies
$\|z-x\|^2+\|w-u\|^2>\|z-y\|^2+\|w-v\|^2$, or equivalently
$$
\|(z,w)-(x,u)\|^2>\|(z,w)-(y,v)\|^2,
$$
which contradicts $P_{{\mathcal L}_{p,q}}(z,w)=(x,u)$, as $(y,v)\in {\mathcal L}_{p,q}$.
\end{proof}
In order to  simplify the notations of our main result,   for a fixed  $z\in\R^p$ and $w\in\R^q$,  we define 
\begin{equation}   \label{eq:fphi}
 f(\lambda):=z-\frac{1}{1+\lambda}\|w\|e+\frac{\lambda}{1+\lambda}\|w\|e^p.
\end{equation}
\begin{theorem}\label{th:maintheo}
Let $(z,w)\in\R^p\times\R^q$, then  the following statements hold:
\begin{itemize}
  \item [\textrm{$(1)$}]
  If $\langle P_{(\R_{\geq+}^p)^*}(-z), e\rangle \geq \|w\|$, then 
  $$
  P_{{\mathcal L}_{p,q}}(z,w)= (P_{\R_{\geq+}^p}(z),0), \quad \qquad P_{{\mathcal L}_{p,q}^*}(-z,-w)=(P_{(\R_{\geq+}^p)^*}(-z),-w);
  $$
  \item [\textrm{$(2)$}]
  If $P_{\R_{\geq+}^p}(z)_p \geq \|w\|$, then 
  $$
  P_{{\mathcal L}_{p,q}}(z,w)=(P_{\R_{\geq+}^p}(z),w), \quad \qquad P_{{\mathcal L}_{p,q}^*}(-z,-w)=(P_{(\R_{\geq+}^p)^*}(-z),0);
  $$
  \item [\textrm{$(3)$}]
  If $\langle P_{(\R_{\geq+}^p)^*}(-z), e\rangle < \|w\|$ and  $P_{\R_{\geq+}^p}(z)_p < \|w\|$,  then there holds 
  \begin{equation}\label{eeee2}
  P_{{\mathcal L}_{p,q}}(z,w)=\left( P_{\R_{\geq+}^p}(f(\lambda))+ \frac{1}{1+\lambda}\|w\|e ,\frac{1}{1+\lambda}w\right), 
  \end{equation}
  \begin{equation}\label{eeee3}
  P_{{\mathcal L}_{p,q}^*}(-z,-w)=\left(   P_{(\R_{\geq+}^p)^*}(-f(\lambda))+\frac{\lambda}{1+\lambda}\|w\|e^ p,-\frac{\lambda}{1+\lambda}w\right), 
  \end{equation}
\end{itemize}
where $\lambda:={\|w\|}/{\langle P_{\mathbb R^{p+1}_{\geq +}}(z,\|w\|),e^{p+1}\rangle}-1$.
\end{theorem}
\begin{proof}
Let $(z,w)\in\R^p\times\R^q$. Our task is to find $(x, u)\in {\mathcal L}_{p,q}$  and $(y, v)\in {\mathcal L}_{p,q}^* $ such that
\begin{equation} \label{eq:eqiv}
P_{{\mathcal L}_{p,q}} (z,w) = (x, u), \qquad P_{{\mathcal L}_{p,q}^*}(-z,-w) = (y, v).
\end{equation}
 To prove item (1),  assume that $\langle P_{(\R_{\geq+}^p)^*}(-z), e\rangle \geq \|w\|$. Thus, by   item~(i) of  Lemma~\ref{le:u0v0} we must  have $u=0$.  Since  $P_{{\mathcal L}_{p,q}} (z,w) = (x, 0)$ and $P_{{\mathcal L}_{p,q}^*}(-z,-w) = (y, v)$,  applying Theorem~\ref{th:mt} for ${\mathcal L}_{p,q}$ we have  $(x,0)\in {\mathcal L}_{p,q}$, $(y,v)\in {\mathcal L}_{p,q}^*$, $\langle (x,0), (y, v)\rangle=0$ and  $(z,w)=(x,0)-(y,v)$. Thus,  $x\in{\R_{\geq+}^p}$ and $y\in ({\R_{\geq+}^p})^*$,  $\langle x, y \rangle=0$, $z=x-y$ and $v=-w$. Now, applying Theorem~\ref{th:mt} for  ${\R_{\geq+}^p}$ we conclude that $x=P_{\R_{\geq+}^p}(z)$ and $y=P_{(\R_{\geq+}^p)^*}(-z)$, which together with \eqref{eq:eqiv}, $u=0$ and $v=-w$  proves  item $(1)$.

We proceed  to prove  item (2).  Since  $P_{\R_{\geq+}^p}(z)_p \geq \|w\|$, the    item~(ii) of  Lemma~\ref{le:u0v0}  implies  $v=0$.  Considering that   $P_{{\mathcal L}_{p,q}} (z,w) = (x, u)$ and $P_{{\mathcal L}_{p,q}^*}(-z,-w) = (y, 0)$,  applying Theorem~\ref{th:mt} for ${\mathcal L}_{p,q}$ we have  $(x,u)\in {\mathcal L}_{p,q}$, $(y,0)\in {\mathcal L}_{p,q}^*$, $\langle (x,u), (y, 0)\rangle=0$ and  $(z,w)=(x,u)-(y,0)$. Hence,   $x\in{\R_{\geq+}^p}$ and $y\in ({\R_{\geq+}^p})^*$,  $\langle x, y \rangle=0$, $z=x-y$ and $u=w$. Using Theorem~\ref{th:mt} for  ${\R_{\geq+}^p}$, we conclude that $x=P_{\R_{\geq+}^p}(z)$ and $y=P_{(\R_{\geq+}^p)^*}(-z)$, which together with \eqref{eq:eqiv}, $v=0$ and $u=w$  yields  item (2).

To prove  item (3), we first note that  conditions $\langle P_{(\R_{\geq+}^p)^*}(-z), e\rangle < \|w\|$ and  $P_{\R_{\geq+}^p}(z)_p < \|w\|$ together with item~(iii) of   Lemma~\ref{le:u0v0} implies that $u\neq 0$ and $v\neq 0$. Moreover, it follows from Theorem~\ref{th:mt}  that  \eqref{eq:eqiv} is equivalent to
\begin{equation} \label{eq:tef}
(x,u,y, v)\in C({\mathcal L}_{p,q}) \qquad   (z,w)=(x,u)-(y,v).
\end{equation}
Due to  $u\neq 0$, $v\neq 0$ and   \eqref{eq:tef}, we  apply Proposition~\ref{prop:csa} to obtain the following equivalent conditions
 \begin{equation} \label{eq:teff}
 x_p = \|u\|, \qquad \langle y, e\rangle=\|v\|,\qquad   \langle u,v\rangle=-\|u\|\|v\|, \qquad (x-\|u\|e,y-\|v\|e^ p)\in C(\R_{\geq + }^p),
 \end{equation}
  \begin{equation} \label{eq:tteff}
  z=x-y, \qquad w=u-v.
  \end{equation}
Since $\langle u,v\rangle=-\|u\|\|v\|$, $u\neq 0$ and  $v\neq 0$, there exists $\lambda>0$ such that $v=-\lambda u$. Hence, it follows from the second equality in  \eqref{eq:tteff} that
 \begin{equation} \label{eq:systeq}
u=\frac{1}{1+\lambda}w, \qquad \quad v=-\frac{\lambda}{1+\lambda}w.
  \end{equation}
Meanwhile, the second equality in \eqref{eq:teff}  gives $\langle y,e\rangle=\|v\|$. Thus we have that
\begin{align}
\langle y,e\rangle=\frac{\lambda}{1+\lambda}\|w\|.\label{eq:p1}
\end{align} 
Since  $(x-\|u\|e,y-\|v\|e^ p)\in C(\R_{\geq + }^p)$ by  \eqref{eq:teff}, applying Theorem~\ref{th:mt} for  ${\R_{\geq+}^p}$ we obtain
$$
x-\|u\|e=P_{\R_{\geq+}^p}\left(x-\|u\|e-y+\|v\|e^ p\right), \qquad  y-\|v\|e^ p= P_{(\R_{\geq+}^p)^*}\left(-x+\|u\|e+y-\|v\|e^ p\right).
$$
Thus, by using the first equality in \eqref{eq:tteff} and \eqref{eq:systeq}, we obtain after some calculations that
\begin{align}
x&=  P_{\R_{\geq+}^p}\left(z-\frac{1}{1+\lambda}\|w\|e+\frac{\lambda}{1+\lambda}\|w\|e^ p\right)+ \frac{1}{1+\lambda}\|w\|e; \label{eq:sx}\\
y&=  P_{(\R_{\geq+}^p)^*}\left(-z+\frac{1}{1+\lambda}\|w\|e-\frac{\lambda}{1+\lambda}\|w\|e^ p\right) +  \frac{\lambda}{1+\lambda}\|w\|e^ p. \label{eq:sy}
\end{align}
Hence,  combining \eqref{eq:eqiv} with \eqref{eq:systeq}, \eqref{eq:sx} and \eqref{eq:sy}  and considering  \eqref{eq:fphi}, we obtain
\eqref{eeee2} and \eqref{eeee3}.  It remains to compute $\lambda$. For that,  by using \eqref{eeee2} we can  apply  Lemma~\ref{eq:mle} with  
$$
x=P_{\R_{\geq+}^p}(f(\lambda))+ \frac{1}{1+\lambda}\|w\|e, \qquad \beta=\frac{1}{1+\lambda}, 
$$
to  conclude     that 
\begin{equation} \label{eq:mlef}
 P_{\R_{\geq+}^{p+1}}(z, \|w\|)=\left(P_{\R_{\geq+}^p}(f(\lambda))+ \frac{1}{1+\lambda}\|w\|e, \frac{1}{1+\lambda} \|w\|\right), 
\end{equation}
which gives $\langle P_{\mathbb R^{p+1}_{\geq  +}}(z,\|w\|),e^{p+1}\rangle=\frac{1}{1+\lambda}\|w\|$. Therefore, 
$\lambda={\|w\|}/{\langle P_{\mathbb R^{p+1}_{\geq +}}(z,\|w\|),e^{p+1}\rangle}-1$, which concludes the proof.
\end{proof}
\begin{remark}
If $p=1$, then the projection formulas in Theorem \ref{th:maintheo} become the projection onto the second-order cone (see Exercise 8.3 (c) in \cite{BoydVandenberghe}).
\end{remark}
The next theorem states that to compute a projection onto the cone ${\R_{\geq+}^p}$ it is sufficient to know how to compute a projection onto the cones $ \R^p_{\geq}$ and ${\mathbb R}^p_+$, its proof can be found in \cite{nemeth2012projec}.  For the sake of completeness we include its proof here.
\begin{theorem} \label{th:nt}
For any  $z\in \R^p$, there holds $P_{\R_{\geq+}^p}(z)=P_{\R_{\geq}^p}(z)^+=P_{{\mathbb R}^p_+}(P_{\R_{\geq}^p}(z))$.
\end{theorem}
\begin{proof}
To simplify the notations set ${\cal K}={\R_{\geq}^p}$. Thus, Theorem~\ref{th:mt} yields  
$$
z=P_{\cal K}(z)-P_{{\cal K}^*}(-z), \qquad  \langle P_{\cal K}(z), P_{{\cal K}^*}(-z)\rangle = 0.
$$
Moreover, as $P_{\cal K}(z)= P_{\cal K}(z)^{+}- P_{\cal K}(z)^{-}$, the last inequality becomes  
\begin{equation} \label{eq:filt}
z=P_{\cal K}(z)^{+}- P_{\cal K}(z)^{-}-P_{{\cal K}^*}(-z), \qquad  \langle P_{\cal K}(z)^{+}- P_{\cal K}(z)^{-}, P_{{\cal K}^*}(-z)\rangle = 0.
\end{equation} 
Note that $P_{\cal K}(z)^{+}\in {\cal K}$ and $- P_{\cal K}(z)^{-}\in {\cal K}$. Indeed, 
due to  $P_{\cal K}(z) \in {\cal K}$ and $- P_{\cal K}(z) \in {\cal K}$, we have from  \eqref{eq:defmca} that    $P_{\cal K}(z)_1\geq P_{\cal
K}(z)_2\geq\cdots\geq P_{\cal K}(z)_p$ and $-P_{\cal K}(z)_1\geq -P_{\cal K}(z)_2\geq\cdots\geq -P_{\cal K}(z)_p$. Hence,  bearing in mind that
the functions $\R \ni t\mapsto t^+$ and  $\R \ni t\mapsto -t^-$ are monotone 
increasing, we also have  $P_{\cal K}(z)^{+}_1\geq P_{\cal K}(z)^{+}_2\geq\cdots\geq P_{\cal K}(z)^{+}_p\geq 0$ and $-P_{\cal K}(z)^{-}_1\geq
-P_{\cal K}(z)^{-}_2\geq\cdots\geq -P_{\cal K}(z)^{-}_p$. Thus, $P_{\cal K}(z)^{+}\in{\R_{\geq +}^p}\subset  {\cal K}$ and $-P_{\cal K}(z)^{-}\in
{\cal K}$. Therefore, the second equality in 
\eqref{eq:filt} yields 
\begin{equation} \label{eq:silt}
\langle P_{\cal K}(z)^{+}, P_{{\cal K}^*}(-z)\rangle= \langle  P_{\cal K}(z)^{-}, P_{{\cal K}^*}(-z)\rangle = 0.
\end{equation} 
On the other hand,  \eqref{eq:defdmca} and \eqref{eq:defdmnc} implies  ${\cal K}^* \subset ({\R_{\geq +}^p})^*$.  Furthermore,  due to $ P_{\cal
K}(z)^{-}\in{\R_{+}^p}$ and $ {\R_{+}^p}\subset ({\R_{\geq +}^p})^*$, we conclude that 
\begin{equation} \label{eq:tilt}
P_{\cal K}(z)^{-}+ P_{{\cal K}^*}(-z) \in  ({\R_{\geq +}^p})^*.
\end{equation} 
Considering \eqref{eq:silt} and  $ \langle P_{\cal K}(z)^{+}, P_{\cal K}(z)^{-}\rangle =0$, we also have
\begin{equation} \label{eq:filtn}
\langle P_{\cal K}(z)^{+}, P_{\cal K}(z)^{-}+ P_{{\cal K}^*}(-z)\rangle =0.
\end{equation} 
Therefore,  the reformulation   
$z=P_{\cal K}(z)^{+}- \left(P_{\cal K}(z)^{-}+P_{{\cal K}^*}(-z)\right)$ 
of \eqref{eq:filt}$_1$, together with the formulas 
$P_{\cal K}(z)^{+}\in {\R_{\geq +}^p}$, \eqref{eq:tilt}, \eqref{eq:filtn} and Theorem~\ref{th:mt} imply that $P_{\cal
K}(z)^{+}=P_{\R_{\geq+}^p}(z)$%
, which is the desired result.
 \end{proof}
We end this section by pointing out that  efficient numerical methods to compute projection onto the cone $ \R^p_{\geq}$ can be found by using the
pool-adjacent-violators algorithm for isotonic regreession 
\cite{BestChakravarti1990, mair2009isotone2009}. 
For projecting onto the cone ${\mathbb R}^p_+$, we only need to apply the formula of Theorem
\ref{th:nt} to the output of the $p$-dimensional PAVA.  

In the next section we present a conic optimisation problem with respect to the MESOC related
to a portfolio optimisation problem. We note that this problem is an adaptation of Xiao's
application in Chapter 4 of his PhD dissertation \cite{Xiao2021} 
(which is an improved version of the application in Section 3 of \cite{NemethXiao2018}) to 
the monotone case. Such problems can be solved by algorithms where the projection
onto the intersection of MESOC with a hyperplane is important \cite{HenrionMalick2012}. Our 
efficient projection method onto MESOC can be incorporated into Dykstra's alternating 
projection method \cite{BoyleDykstra1986} for the aforementioned intersection. One can also investigate a possible more direct adaptation of our method 
to such projections.
\section{An application of the monotone extended second-order cone to portfolio optimisation} \label{eq:apo}
 Markowitz developed the mean-variance (MV) model in \cite{markowitz1959portfolio}, which is the classical method in investigating the problem of portfolio optimisation. Suppose we build portfolio by using $n$ arbitrary assets. Let $w\in\R^n$ denote the weights of the assets, $r\in\R^n$ represent the return of assets and $\Sigma\in\R^n\times\R^n$ be the covariance matrix. Then, the two traditional and equivalent MV models could be given as:
\[
\min_w\left\{w^{\top}\Sigma w:\,r^{\top}w\geq \alpha,\,e^{\top}w=1 \right\}
\]
and
\[
\max_w\left\{r^{\top}w:\,w^{\top}\Sigma w\leq\beta,\,e^{\top}w=1 \right\},
\]
where $\alpha$ is the minimum profit that the investor demands and $\beta$ is the minimum risk that the investor wants to tolerate. They are typical quadratic optimisation problems with higher computational complexity.

In order to reduce the complexity of solving the portfolio optimisation problem, based on the traditional mean-variance model, Konno and Yamazaki developed the mean-absolute deviation (MAD) model in \cite{konno1991mean}, by replacing the risk measure from the covariance matrix to the absolute deviation. They demonstrated that the results obtained by using MAD model are similar with the results obtained by using the MV model when the return of assets are multivariate normally distributed. It has also been recognized that the MAD model has reduced the computational complexity significantly \cite{konno2005mean,konno1999mean}. Before introducing the MAD model, we will give the definitions of some key parameters.

Denote the returns of assets be $\tilde{r} = (\tilde{r_1},\ldots,\tilde{r_n})^{\top} \in \R^n$. Suppose that they are distributed over a finite sequence of points $R^j = \left(R_{1}^j,\ldots,R_{n}^j\right)^{\top}\in\R^n$, where $j=1,\ldots,T$ and $R^j$ denotes $T$ different scenarios such that the behaviour of the assets are different in different scenarios. Meanwhile, denote $f_j$ the probability distribution of the rates of returns of assets, that is
\[
f_j=\textrm{Probability}\left\{ (\tilde{r_1},\ldots,\tilde{r_n})^{\top}=\left(R_{1}^j,\ldots,R_{n}^j\right)^{\top}\right\},\qquad j=1,\ldots,T.
\]
The sequences $\{R^j\}_{j=1,\ldots,T}$ and $\{f_j\}_{j=1,\ldots,T}$ can be obtained by using the historical data of assets and some techniques for the
future projection of these assets. Meanwhile, since $f_j\in [0,1]$, $j=1,\ldots,T$ represent probabilities, we will have $e^\top f=1$, where $f=(f_1,\ldots,f_T)$.
In particular, we have
\[
r=\mathbb{E}[\tilde{r}]=f_1R^1+\cdots +f_TR^T.
\]
In order to measure the uncertainty of the returns of the assets for
$j=1,\ldots,T$, let us define $U=(U_1,\ldots,U_T)^\top$, where $U_j=R^j-r$. Let
$y_j$ denote the upper bound of disturbance of return at day $j$. Then, the
traditional MAD model can be represented as the following linear programming problem:
\begin{equation*}
\begin{aligned}
& \underset{y,w}{\text{min}}
& & c_0f^{\top}y-r^{\top}w \\
& \text{s.t.}
& & y_j\geq|U^{\top}_jw|, \quad j=1,\ldots,T,\\
&&& e^{\top}w=1,
\end{aligned}
\end{equation*}
where $c_0>0$ is the Arrow-Pratt absolute risk-aversion index defined in \cite{kallberg1984mis}. 

In reality, the uncertainty of the returns of the assets will increase with the increasing of the investment horizon. Thus, it is meaningful to optimize the MAD model to make it more in line with the real-world market behaviour. Meanwhile, by using Cauchy's inequality, we also have $|U^{\top}_jw|\leq \|U_j\|\|w\|$ for any $j$. Then, based on the current MAD model, we obtain the following related problem
\begin{equation*}
\begin{aligned}
& \underset{y,w}{\text{min}}
& & c_0f^{\top}y-r^{\top}w \\
& \text{s.t.}
& & y_T\geq y_{T-1}\geq\ldots\geq y_1\geq \|U_{j^*}\|\|w\|,\\
&&& e^{\top}w=1,
\end{aligned}
\end{equation*}
where $j^*=\argmin_j|U^{\top}_jw|$, for $j=1,\ldots,T$.
Note that the vector \[\left(\frac{y_T}{\|U_{j^*}\|},\frac{y_{T-1}}{\|U_{j^*}\|},\ldots,\frac{y_1}{\|U_{j^*}\|},w\right)^{\top}\] belongs to the monotone extended second-order cone  ${\mathcal L}_{T,n}$. Thus, the last problem is equivalent to the following conic optimization problem:

\begin{equation*}
\begin{aligned}
& \underset{y,u}{\text{min}}
& & c_0f^{\top}y-r^{\top}\frac{u}{\|U_{j^*}\|} \\
& \text{s.t.}
& & e^{\top}u=\|U_{j^*}\|,\\
&&& \left(y_T,y_{T-1},\ldots,y_1,u\right)^{\top}\in {\mathcal L}_{T,n},  
\end{aligned}
\end{equation*}

where $u:=w\|U_{j^*}\|$.
\section{Final remarks}
In this  paper we have introduced the monotone extended second-order cone and its 
dual cone. We have reduced the projection onto the MESOC to two isotonic regreessions in 
neighboring dimensions. The isotonic regression can be solved efficiently by 
pool-adjacent-violators algorithm \cite{BestChakravarti1990, mair2009isotone2009}. We 
have also presented an application of the MESOC to portfolio optimization via a conic 
optimization problem related to the mean-absolute deviation model \cite{konno1991mean}. 
Knowing the projection onto the MESOC can be a useful ``ingredient'' of projection methods 
for the latter problem. We predict more direct applications of the projection onto MESOC
to practical problems. These applications would be regressions with respect to 
a set of points whose distance (more generally a ``cost'') from a source point is expected to
decrease and only the position of the point closest to the source is important. For 
example 
to capture a strong enough ``signal'' of a point from the source it is expected to put the
better capturing devices further from the source. If one point is (significantly) closer to 
the source than the other ones, than its position becomes important, because any obstacle 
``between'' this point and the source will have a dominant impact in comparison to the other 
points.
In this probably a better device would be needed than one based on the distance from the 
source only. Similar
types of problems can be imagined in case of a football (i.e., soccer) game where one would expect
the defenders to be in general further from the opponents goal and the 
striker's position to be much more important.  


\begin{thebibliography}{10}

\bibitem{FerreiraNemeth2018}
Ferreira~OP, N\'{e}meth~SZ. How to project onto extended second-order cones. J
  Global Optim. 2018;\hspace{0pt}70(4):707--718.

\bibitem{mair2009isotone2009}
	de~Leeuw~J, Hornik~K, Mair~P.  Isotone optimization in R:
  pool-adjacent-violators algorithm (PAVA) and active set methods. Journal of
  statistical software. 2009;\hspace{0pt}32(5):1--24.

\bibitem{nemeth2012projec}
N{\'e}meth~AB, N{\'e}meth~SZ. How to project onto the monotone nonnegative cone
  using pool adjacent violators type algorithms. arXiv preprint, 12012343v2.
  2012;\hspace{0pt}:1--6.

\bibitem{BestChakravarti1990}
Best~MJ, Chakravarti~N. Active set algorithms for isotonic regression; a
  unifying framework. Math Programming. 1990;\hspace{0pt}47(3, (Ser.
  A)):425--439.

\bibitem{le2016application}
Le~LT, Priestley~JL. Application of isotonic regression in predicting business
  risk scores. Published and Grey Literature from PhD Candidates3.
  2016;\hspace{0pt}.

\bibitem{2Nemeth2016}
N\'{e}meth~AB, N\'{e}meth~SZ. Isotonic regression and isotonic projection.
  Linear Algebra Appl. 2016;\hspace{0pt}494:80--89.

\bibitem{MR1655138}
Lobo~MS, Vandenberghe~L, Boyd~S, et~al. Applications of second-order cone
  programming. Linear Algebra Appl. 1998;\hspace{0pt}284(1-3):193--228. ILAS
  Symposium on Fast Algorithms for Control, Signals and Image Processing
  (Winnipeg, MB, 1997).

\bibitem{MR3158056}
Gajardo~P, Seeger~A. Equilibrium problems involving the {L}orentz cone. J
  Global Optim. 2014;\hspace{0pt}58(2):321--340.

\bibitem{MR2377196}
Kong~L, Xiu~N, Han~J. The solution set structure of monotone linear
  complementarity problems over second-order cone. Oper Res Lett.
  2008;\hspace{0pt}36(1):71--76.

\bibitem{MR2116450}
Malik~M, Mohan~SR. On {$\bf Q$} and {${\bf R}_0$} properties of a quadratic
  representation in linear complementarity problems over the second-order cone.
  Linear Algebra Appl. 2005;\hspace{0pt}397:85--97.
 

\bibitem{MR3010551}
	Zhang~LL, Li~JY, Zhang~HW, Pan~SH. A second-order cone complementarity approach
  for the numerical solution of elastoplasticity problems. Comput Mech.
  2013;\hspace{0pt}51(1):1--18.
  

\bibitem{MR2925039}
Yonekura~K, Kanno~Y. Second-order cone programming with warm start for
  elastoplastic analysis with von {M}ises yield criterion. Optim Eng.
  2012;\hspace{0pt}13(2):181--218.


\bibitem{MR2568432}
Luo~GM, An~X, Xia~JY. Robust optimization with applications to game theory.
  Appl Anal. 2009;\hspace{0pt}88(8):1183--1195.


\bibitem{MR2522815}
Nishimura~R, Hayashi~S, Fukushima~M. Robust {N}ash equilibria in {$N$}-person
  non-cooperative games: uniqueness and reformulation. Pac J Optim.
  2009;\hspace{0pt}5(2):237--259.

\bibitem{KCY2011}
Ko~CH, Chen~JS, Yang~CY. Recurrent neural networks for solving second-order
  cone programs. Neurocomputing. 2011;\hspace{0pt}74:3464--3653.

\bibitem{MR2179239}
Chen~JS, Tseng~P. An unconstrained smooth minimization reformulation of the
  second-order cone complementarity problem. Math Program.
  2005;\hspace{0pt}104(2-3, Ser. B):293--327.


\bibitem{konno1991mean}
Konno~H, Yamazaki~H. Mean-absolute deviation portfolio optimization model and
  its applications to tokyo stock market. Management science.
  1991;\hspace{0pt}37(5):519--531.

\bibitem{FacchineiPang2003-I}
Facchinei~F, Pang~JS. Finite-dimensional variational inequalities and
  complementarity problems. {V}ol. {I}. Springer-Verlag, New York; 2003.
  Springer Series in Operations Research.

\bibitem{MR0139919}
Moreau~JJ. D\'ecomposition orthogonale d'un espace hilbertien selon deux
  c\^ones mutuellement polaires. C R Acad Sci Paris.
  1962;\hspace{0pt}255:238--240.

\bibitem{Niculescu2017}
Niculescu~CP, St\u{a}nescu~MM. A note on {A}bel's partial summation formula.
  Aequationes Math. 2017;\hspace{0pt}91(6):1009--1024.


\bibitem{Niculescu2018}
Niculescu~CP, Persson~LE. Convex functions and their applications. Springer,
  Cham; 2018. CMS Books in Mathematics/Ouvrages de Math\'{e}matiques de la SMC;
  a contemporary approach, Second edition of [ MR2178902];


\bibitem{BoydVandenberghe}
Boyd~S, Vandenberghe~L. Convex optimization. Cambridge: Cambridge University
  Press; 2004.

\bibitem{Xiao2021}
Xiao~L. Complementarity and related problems. arXiv preprint, 210807412.
  2021;\hspace{0pt}:1--157.

\bibitem{NemethXiao2018}
N\'{e}meth~SZ, Xiao~L. Linear complementarity problems on extended second-order
  cones. J Optim Theory Appl. 2018;\hspace{0pt}176(2):269--288.
  

\bibitem{HenrionMalick2012}
Henrion~D, Malick~J. Projection methods in conic optimization. In: Handbook on
  semidefinite, conic and polynomial optimization. (Internat. Ser. Oper. Res.
  Management Sci.; Vol. 166). Springer, New York; 2012. p. 565--600.


\bibitem{BoyleDykstra1986}
Boyle~JP, Dykstra~RL. A method for finding projections onto the intersection of
  convex sets in {H}ilbert spaces. In: Advances in order restricted statistical
  inference ({I}owa {C}ity, {I}owa, 1985). (Lect. Notes Stat.; Vol.~37).
  Springer, Berlin; 1986. p. 28--47.


\bibitem{markowitz1959portfolio}
Markowitz~H. Portfolio selection [reprint of {J}. {F}inance {\bf 7} (1952), no.
  1, 77--91]. In: Financial risk measurement and management. (Internat. Lib.
  Crit. Writ. Econ.; Vol. 267). Edward Elgar, Cheltenham; 2012. p. 197--211.

\bibitem{konno2005mean}
Konno~H, Koshizuka~T. Mean-absolute deviation model. Iie Transactions.
  2005;\hspace{0pt}37(10):893--900.

\bibitem{konno1999mean}
Konno~H, Wijayanayake~A. Mean-absolute deviation portfolio optimization model
  under transaction costs. J Oper Res Soc Japan.
  1999;\hspace{0pt}42(4):422--435.


\bibitem{kallberg1984mis}
Kallberg~JG, Ziemba~WT. Mis-specifications in portfolio selection problems. In:
  Risk and capital. Springer; 1984. p. 74--87.

\end{thebibliography}

\end{document}